\documentclass[12pt,letterpaper, reqno]{amsart}

\usepackage{amssymb,latexsym, amsmath, amsxtra}
\usepackage[dvips]{graphics}

\theoremstyle{plain}
\newtheorem{theorem}{Theorem}[section]
\newtheorem{corollary}[theorem]{Corollary}

\newtheorem{lemma}[theorem]{Lemma}
\newtheorem{proposition}[theorem]{Proposition}
\theoremstyle{definition}

\theoremstyle{remark}

\numberwithin{equation}{section}
\numberwithin{theorem}{section}
\numberwithin{table}{section}
\numberwithin{figure}{section}

\newcommand{\D}{\mathbb D}

\newcommand{\C}{\mathbb C}

\def\({\left(}
\def\){\right)}

\begin{document}
\title[Approximations with zeros on a circle]{Approximation by polynomials and  Blaschke products having all zeros on a circle}
\author{David W. Farmer and Pamela Gorkin}

\address{
{\parskip 0pt
American Institute of Mathematics\endgraf
farmer@aimath.org\endgraf
\null
Department of Mathematics\endgraf
Bucknell University\endgraf
pgorkin@bucknell.edu\endgraf
}
  }


\begin{abstract}
We show that a nonvanishing analytic function on a domain in the
unit disc can be approximated by (a scalar multiple of) a
Blaschke product whose zeros lie on a prescribed circle enclosing the domain.
We also give a new proof of the analogous classical result for polynomials.
A connection is made to universality results for the Riemann zeta function.
\end{abstract}

\maketitle

\section{Introduction}

While every analytic function on a domain can be approximated pointwise on
the domain by a polynomial,  there are many other interesting questions
that can be asked about the approximating polynomials; for example,
can the sequence of approximating polynomials be chosen to be uniformly
bounded? In this paper, we focus on the case in which the domain is the
open unit disc and we ask where the zeros of the polynomials will lie.

When our attention is focused on bounded analytic functions, there is
another class of functions that can be used to approximate our bounded
analytic function: the set of Blaschke products. A finite Blaschke
product is a function of the form 
\begin{equation}
B(z) = \lambda \prod_{j = 1}^N
\frac{z - a_j}{1 - \overline{a_j} z},
\end{equation}
where $|a_j| < 1$ for all $j$
and $|\lambda| = 1$. Caratheodory's theorem (see \cite[p. 6]{G}, for
example) shows that if $f$ is an analytic function defined on the open
unit disc, $\mathbb{D}$, and $f$ is bounded by $1$ in modulus, then there
is a sequence $\{B_k\}$ of finite Blaschke products converging to $f$
pointwise on $\mathbb{D}$. Again, it is certainly interesting to ask
where the zeros of the approximating Blaschke products may lie.

Looking at more general domains, one natural question is the following:
Given a holomorphic function in a Jordan region, when can it be
approximated by a polynomial with zeros lying on the boundary? This
question was answered by G. MacLane \cite{GM} in 1949. Curiously,
many texts dealing with the study of polynomials or approximation
by polynomials do not include reference to this work (\cite{M},
\cite{SS}), though three different proofs of this result are given by Korevaar \cite{K}.

MacLane's work focused on showing that a zero-free holomorphic
function can be approximated by a polynomial with zeros on the boundary,
when the boundary satisfies certain smoothness conditions. C. Chui
 \cite{C}, \cite{C1} looked at the problem of bounded approximation
of a zero-free bounded holomorphic function by what he called
$C$-polynomials. Chui showed that every zero-free bounded holomorphic
function defined on $\mathbb{D}$, can be boundedly approximated by
polynomials with zeros lying on the unit circle. In 1968, Z. Rubinstein
showed that given a zero-free holomorphic function with $f(0) = 1$, there
exists a sequence of $C$-polynomials mapping $0$ to $1$ that converges
to $f$ uniformly on every compact set in $\mathbb{D}$. In addition,
when the function $f$ is bounded, the sequence converges boundedly.
One natural approach, that of looking at the zeros of the partial sums of a series,
has been further studied by Korevaar and others. This and related work can be found in
 \cite{K1}, \cite{K3}, and \cite{K2}.

In this paper, we will focus on polynomials as well as Blaschke
products. In general, a Blaschke product is a function of the form 
\begin{equation}
B(z)
= \lambda z^n \prod_{j = 1}^\infty \frac{-\overline{a_j}}{|a_j|}\frac{z
- a_j}{1 - \overline{a_j} z},
\end{equation}
where $n$ is a nonnegative integer,
$0< |a_j| < 1$, $|\lambda| = 1$, and $\sum_j (1 - |a_j|) < \infty$.
As indicated above, when there are finitely many
$a_j$,  the function is said to be a finite Blaschke product; the term
Blaschke product is used when the distinction between finite and infinite
products is unimportant. The Blaschke product is determined by its zeros
and every bounded analytic function on $\mathbb{D}$ can be written as
a product of a Blaschke product and an analytic function with no zeros
on $\mathbb{D}$. In fact, every bounded analytic function $f$ has what is known as an inner-outer factorization;
that is, there exists an analytic function $I$, bounded by $1$ in
the unit disc, with $|I| = 1$ almost everywhere on the unit circle,
and an outer function $g$ with the property that $f = I g$. The inner
function $I$ can be further factored into a (possibly infinite) Blaschke
product, defined with the zeros of the function $f$, and a singular
inner function. Though the singular inner function will have no zeros
in $\mathbb{D}$, Frostman's theorem \cite[p. 79]{G}, shows that it can
be uniformly approximated by a Blaschke product. Therefore, every inner
function can be uniformly approximated by Blaschke products.

R. Douglas and W. Rudin \cite{DR} showed that such results can be
applied to another very important algebra of functions. In fact,
as they showed, every essentially bounded measurable function $u$
with modulus $1$ a.e. on the unit circle can be approximated using Blaschke products and
their conjugates. More precisely, given $\varepsilon > 0$, there exist
Blaschke products $b_1$ and $b_2$ such that $\|u - b_1 \overline{b_2}\|
< \varepsilon$. Several interesting results on these algebras followed,
and then, in 1982, P. Jones \cite{PJ} showed that the Blaschke products
above could be chosen to be interpolating Blaschke products; that is,
the zero sequence $\{z_n\}$ of the Blaschke product has the property
that given any bounded sequence $\{w_n\}$ of complex numbers, there
exists a bounded analytic function $f$ such that $f(z_n) = w_n$. A
characterization of such sequences due to L. Carleson~\cite{Ca} makes
such sequences quite easy to deal with. In particular,  the zeros of the
interpolating Blaschke products are separated in an extremely useful way,
and as a consequence interpolating Blaschke products have many desirable
properties. In trying to understand these results, Jones and Garnett
\cite[p. 430]{G} asked whether every Blaschke product can be uniformly
approximated by interpolating Blaschke products. While many interested
related results exist (see, for example, \cite{HN} and \cite{GMN} for
recent results ) this problem remains open. Our interest in controlling
the zeros of approximating Blaschke products is directly related to this
very important question. If we can control the placement of the zeros
of an approximating Blaschke product, we have a chance of constructing
approximating interpolating Blaschke products.

Our main result (Corollary \ref{cor:blaschke0})
is that given an
analytic function $g$ that has no zeros
in a neighborhood of $\{z: |z| \le r\}$,
for all $\varepsilon > 0$ and $\delta > 0$ 
there exists
a constant $c_B$ and a finite
Blaschke product $B$ with all zeros on the circle $\{z: |z| = r\}$ such
that 
\begin{equation}
|g(z) - c_B B(z)| < \varepsilon
\end{equation}
on 
$\{z: |z| < r - \delta \}$. Our approach also provides a
new and relatively simple proof of
the fact that a nonvanishing analytic function can
be approximated
uniformly on compacta
by polynomials having zeros on a prescribed circle.

The paper is organized as follows.  In Section~\ref{sec:around0}
we consider
the special case of discs centered at~$0$.   Then in
Section~\ref{sec:pseudohyperbolic} we use the special case to prove more
general cases. In Section~\ref{sec:rmt} we discuss the relationship
with universality results for the
Riemann zeta function and its connection to random matrix theory.

\section{Approximation around 0}\label{sec:around0}

Here is a very simple proof that a polynomial $p$ that does not vanish on a neighborhood of the closure of the unit disc can be approximated uniformly on compacta by polynomials with zeros on the unit circle:
 
Note that if the degree of $p$ is $m$, then the polynomial $p^\star(z)
= z^m \overline{p(1/\overline{z})}$ has all zeros inside the unit disc. Let
\begin{equation}
B(z) = p^\star(z)/p(z).
\end{equation}
Then $B$ is analytic in $\mathbb{D}$,
continuous on the unit circle, maps $\mathbb{D}$ to itself, the
unit circle to itself, and the complement of the closed disc to
itself. Therefore $B$ is a Blaschke product. Now for $k \in \mathbb{N}$,
the set of points in $\overline{\mathbb{D}}$ for which $B(z) = z^{-k}$
lie on the unit circle.  Therefore, the polynomial 
\begin{equation}
p(z) + z^k
p^\star(z)
\end{equation}
has all its zeros on the unit circle and approximates $p$
on compacta as $k \to \infty$. This is, more or less, the proof given
by Z. Rubinstein \cite{R} in 1968. However,
it does not seem possible to adapt this proof to the case
of Blaschke products.
We present an alternate proof of this result, before turning
to our theorem on Blaschke products.

First we consider the case of approximating on a disc centered
at~0.

\begin{theorem}\label{thm:around0}
Suppose $f$ is analytic and nonvanishing in a neighborhood of
$|z|\le r$.  Then there exist numbers $|\xi_j|=|\eta_j|=1$,
$A\in \C$,
and positive integers $\nu(j)$
such that
\begin{equation}
f(z)=
A \prod_{j=1}^\infty (1+ \xi_j z^j)^{\nu(j)}
(1+\eta_j z^j),
\end{equation}
for $|z|<\min\{r,1\}$, with the convergence uniform on
$|z|<\min\{r,1-\delta\}$ for any~$\delta>0$.

In particular, $f$ can be approximated on $|z|<\min\{r,1-\delta\}$
by polynomials having all roots on the unit circle.
\end{theorem}

\begin{theorem}\label{thm:denominator0}
Under the same conditions as Theorem~\ref{thm:around0}, if $R<1$ then
\begin{equation}
f(z)=
A \prod_{j=1}^\infty \left(\frac{1+ \xi_j z^j}{1+ \xi_j R^j z^j}\right)^{\nu(j)}
\frac{1+\eta_j z^j}{1+\eta_j R^j z^j},
\end{equation}
for $|z|<\min\{r,1\}$, with the convergence uniform on
$|z|<\min\{r,1-\delta\}$ for any~$\delta>0$.
\end{theorem}

\begin{corollary}\label{cor:blaschke0} Suppose $g$ is analytic and nonvanishing in a neighborhood of
$|z|\le r<1$. 
For all $\varepsilon>0$ and $\delta>0$ there exists a
constant $c_B$ and a
Blaschke product $B$ having all zeros on $|z|=r$
such that
\begin{equation}
\left|g(z)- c_B B(z)\right|  < \varepsilon
\end{equation}
for $|z|<r-\delta$.
\end{corollary}

Before giving the proof of
Theorems~\ref{thm:around0} and~\ref{thm:denominator0}, we describe the
construction in a context that avoids the issue of convergence.

\subsection{Formal power series as infinite products}\label{sec:formal}

\begin{proposition}
\label{prop:formalproduct}
Suppose
\begin{equation}
f(z)=\sum_{n=0}^\infty a_n z^n 
\end{equation}
is a formal power series with $a_n\in \C$ and $a_0 \not = 0$.
We can write
\begin{equation}
f(z)= 
a_0 \prod_{j=1}^\infty (1+ \xi_j z^j)^{\nu(j)}
(1+\eta_j z^j),
\end{equation}
where $|\xi_j|=|\eta_j|=1$ and $\nu(j)$ is a nonnegative integer.
\end{proposition}

\begin{proposition}
\label{prop:formalblaschke}
Suppose
\begin{equation}
f(z)=\sum_{n=0}^\infty a_n z^n
\end{equation}
is a formal power series with $a_n\in \C$ and $a_0 \not = 0$,
and let $R<1$.
We can write
\begin{equation}
f(z)=
a_0 \prod_{j=1}^\infty
\left(\frac{1+ \xi_j z^j}{1+  {\xi}_j R^j z^j}\right)^{\nu(j)}
\frac{1+\eta_j z^j}{1+\eta_j R^j z^j},
\end{equation}
where $|\xi_j|=|\eta_j|=1$ and $\nu(j)$ is a non-negative integer.
\end{proposition}

\begin{corollary}\label{cor:formalblaschke0} Suppose $g$ is analytic
in a neighborhood of
$0$ with $g(0)\not=0$, and suppose $0<r<1$.  Then for all $J>0$ there exists a
constant $c_J$ and a
Blaschke product $B_J$ having all zeros on $|z|=r$ 
such that
\begin{equation}
g(z)- c_J B_J(z) = O(z^J),
\end{equation}
as $z\to 0$.
\end{corollary}

The proof of Corollary~\ref{cor:formalblaschke0} is the same as the
proof of Corollary~\ref{cor:blaschke0}. The proof of latter will be
presented in Section~\ref{proofoftheoremandcorollary}, following the
proof of Theorem~\ref{thm:denominator0}.

We define the product representations inductively, using the following
Lemma.

\begin{lemma}\label{lem:annulus} Let $R_0>0$.
Every $w\in \C$ can be written as
$w=m \xi + \eta$ where $|\xi|=|\eta|=R_0$ and $m$ is a positive
integer.  There are four such representations if $|w|> R_0$,
two representations if $0<|w|\le R_0$,
and infinitely many if $w=0$.
\end{lemma}

\begin{proof}
We need only consider the case $R_0=1$.

Consider the unit circle centered at $m\, e^{i \theta}$.
As $\theta$ goes from $0$ to $2\pi$, that circle sweeps out
an annulus whose inner and outer radii differ by~2.
Every point in the interior of that annulus has two representations
of the form $m\, e^{i \theta} + \eta$.  As $m$ varies, those
annuli cover the complex plane. Every point lies in 
the interior of two
of the annuli, except the points $|w|<R_0$ or whose absolute value
is an integer.
\end{proof}

\begin{proof}[Proof of Proposition \ref{prop:formalproduct}]
We may assume $a_0=1$. By Lemma~\ref{lem:annulus} with $R_0=1$,
we can choose
$\xi_1$, $\eta_1$, and $\nu(1)$ so that
$a_1=\nu(1) \xi_1 + \eta_1$. 

Set
\begin{equation}
P_1(z)=(1+ \xi_1 z)^{\nu(1)}
(1+\eta_1 z)
\end{equation}
and note that $P_1(z)=1+a_1 z + O(z^2)$.  Therefore,
\begin{equation}
\frac{f(z)}{P_1(z)} = \sum_{j=0}^\infty b_j z^j,
\end{equation}
where $b_0=1$ and $b_1=0$. 
Now choose $\xi_2$, $\eta_2$, and $\nu(2)$ so that
$b_2=\nu(2) \xi_2 + \eta_2$.  Setting
\begin{align}
P_2(z)=\mathstrut &(1+ \xi_2 z^2)^{\nu(2)} (1+\eta_2 z^2) \cr
=\mathstrut & 1+ b_2 z^2 + O(z^3)
\end{align}
we have
\begin{equation}
P_1(z)P_2(z) = 1+a_1 z + a_2 z^2 + O(z^3) .
\end{equation}
Proceeding inductively we obtain $f(z) = \prod_j P_j(z)$.
\end{proof}

\begin{proof}[Proof of Proposition \ref{prop:formalblaschke}]
Note that
\begin{equation}
\frac{1+ \xi_j z^j}{1+  {\xi}_j R^j z^j}
= 1+(1-R^j)\xi_j z^j + O(z^{2 j}) .
\end{equation}
So everything goes through in the previous proof, with the
modification that we apply Lemma~\ref{lem:annulus} with $R_0=1-R^j$
\end{proof}

The proofs given above are just formal calculations, and it is not
clear what convergence properties the infinite products might have.
Even if $f$ represents an analytic function in a neighborhood
of the origin, the products can only converge where $f$ does
not vanish.  The convergence of the products will depend on
the growth of the numbers $\nu(j)$, and the above constructions
do not appear to shed light on this.

In the next section we organize the proof in a different way that
may appear more cumbersome, but it gives information
about the analytic properties of the infinite product.

\subsection{Proof of Theorem~\ref{thm:around0}}\label{sec:proof0}

\begin{proof}[Proof of Theorem~\ref{thm:around0}]
Let
\begin{equation}
\frac{f'}{f}(z)=\sum_{n=0}^\infty a_n z^n .
\end{equation}
Since $f'/f$ is analytic in a disc slightly larger than
$|z|<r$, there is a $C>0$ and $\kappa>1/r$ so that
\begin{equation}
|a_n| < C \kappa^n
\end{equation}
for all~$n$.  If $r \ge 1$ we set $\kappa=1+\delta$ for some $\delta>0$.
In particular, $\kappa>1$.

Let
\begin{equation}\label{eqn:g}
g(z)=g_J(z)=\prod_{j=1}^J (1+\xi_j z^j)^{\nu(j)}
(1+\eta_j z^j),
\end{equation}
where $|\xi_j|=|\eta_j|=1$ and $\nu(j)$ is a non-negative integer.
We will choose those parameters so that the first $J$ terms in the
Taylor series for $g'/g$ match those of~$f'/f$.

We have
\begin{align}\label{eqn:gpg}
\frac{g'}{g}(z) 
=\mathstrut & 
\sum_{1\le j \le J}\left( j \nu(j) \xi_j z^{j-1} \frac{1}{1+\xi_j z^j}
+j \eta_j z^{j-1} \frac{1}{1+\eta_j z^j}\right) \cr
=\mathstrut & 
- z^{-1} \sum_{1\le j \le J} j \sum_{m=0}^\infty \biggl(  \nu(j)
(-1)^{m+1} \xi_j^{m+1} z^{j(m+1)} 
+ (-1)^{m+1} \eta_j^{m+1} z^{j(m+1)}\biggr) \cr
=\mathstrut &
- \sum_{k=0}^\infty  z^k
 \sum_{\substack{1\le j \le J\\ j|(k+1)}}
(-1)^{\frac{k+1}{j}} j
	\left( \nu(j) \xi_j^{\frac{k+1}{j}} + \eta_j^{\frac{k+1}{j}} \right)
\\
=\mathstrut & \sum_{k=0}^\infty b_k z^k, \nonumber
\end{align}
say.  On the third line we changed summation index $k=j(m+1)-1$, and
the notation $n|m$, read ``$n$ divides $m$,'' means that $m/n$ is an
integer.

Now we show how to choose the parameters in $g$ to match the Taylor
series coefficients of the logarithmic derivatives.

By Lemma~\ref{lem:annulus} we can choose a non-negative
integer $\nu(1)$
and complex numbers  $|\xi_1|= |\eta_1|=1$ so that
$-a_0=\nu(1) \xi_1 + \eta_1$.
Thus, $b_0=a_0$, and we have matched the first Taylor series
coefficients of $f'/f$ and $g'/g$.

Since the only positive integer that divides $0+1=1$ is $1$,
we will continue to have $b_0=a_0$
no matter what we later choose for $\xi_j$, $\eta_j$, and $\nu(j)$ for
$j\ge 2$.  Likewise, once we have chosen
$\xi_j$, $\eta_j$, and $\nu(j)$ for $j\le K$ so that
$a_j=b_j$ for $j\le K-1$, we will continue to have
$a_j=b_j$ for $j\le K-1$ because if $j|(k+1)$ then $j\le k+1$.

To choose $\xi_j$, $\eta_j$, and $\nu(j)$ for $j\ge 2$, we
have to deal with the fact that the
$j=1$ terms make a contribution to all of those coefficients. Similarly,
the $j=2$ terms contribute to all of the later even-index coefficients,
and so on.  Breaking the sum defining $b_K$ into two parts, we find
\begin{align} \label{eqn:Kplus1} 
b_K=\mathstrut &-\sum_{\substack{1\le j \le K\\ j|(K+1)}}
(-1)^{\frac{K+1}{j}} j
        \left( \nu(j) \xi_j^{\frac{K+1}{j}} + \eta_j^{\frac{K+1}{j}} \right)
-
\sum_{\substack{K+1\le j \le J\\ j|(K+1)}}
(-1)^{\frac{K+1}{j}} j
        \left( \nu(j) \xi_j^{\frac{K+1}{j}} + \eta_j^{\frac{K+1}{j}} \right) \cr
=\mathstrut &
-\sum_{\substack{1\le j \le K\\ j|(K+1)}}
(-1)^{\frac{K+1}{j}} j
        \left( \nu(j) \xi_j^{\frac{K+1}{j}} + \eta_j^{\frac{K+1}{j}} \right)
+
(K+1)
        \Bigl( \nu(K+1) \xi_{K+1} + \eta_{K+1} \Bigr) .
\end{align}
The terms in the sum on the second line of \eqref{eqn:Kplus1} have already been chosen,
so we can use
Lemma~\ref{lem:annulus} to choose $\nu(K+1)$, $\xi_{K+1}$, and $\eta_{K+1}$
so that $b_K=a_K$.

Proceeding in this way we match the first $J$ coefficients of
the logarithmic derivatives.

It remains to bound $\nu(j)$ so that we can bound the tail of~\eqref{eqn:gpg}.

By~\eqref{eqn:Kplus1} and the fact that $b_K=a_K$ we have
\begin{equation}\label{eqn:nukinequality}
(K+1)\nu(K+1) \le  K+1+ |a_K| +
\sum_{\genfrac{}{}{0pt}{}{j|(K+1)}{1\le j\le K}}
\bigl(j \nu(j) + j\bigr) .
\end{equation}

By Lemma~\ref{lem:boundnu}, given at the end of this section, 
the above estimate implies that there exists 
$C'$ so that $n\nu(n) \le C' \kappa^n$ for all $n\ge 1$.
This is sufficient to estimate the tail
for $|z|< 1/\kappa$ because
the coefficient of $z^n$ in \eqref{eqn:gpg} is bounded by
\begin{equation}
|b_n|\le \sum_{j\le J} \bigl( j \nu(j) + j \bigr) \ll J^2 \kappa^J ,
\end{equation}
where we use $\ll$ to mean the quantity is bounded by a constant times $J^2 \kappa^J$.
So
\begin{equation}
\sum_{n\ge J+1} |b_n| |z^n| \ll J^2 \kappa^J \sum_{n\ge J+1} |z|^n
\ll \frac{J^2}{1-|z|} \kappa^J |z|^J,
\end{equation}
which goes to 0 as $J\to\infty$ because $|z|<1/\kappa < 1$.

This shows that $g'(z)/g(z)$ is close to $f'(z)/f(z)$ for
$|z|<1/\kappa$.  We can antidifferentiate
using Cauchy's theorem, so $\log(f)$ is close to $\log(g_J)+c_J$
for $|z|<1/\kappa$, for
some constant~$c_J$.  Now exponentiate 
to get that $f(z)$ is close to $e^{c_J} g(z)$.
Since $g(0)=1$, choose $c_J=\log(f(0))$.

This completes the proof of Theorem~\ref{thm:around0}.
\end{proof}

\begin{lemma}\label{lem:boundnu}
Suppose $\nu(j)\ge 0$ for $j\ge 0$.
If there exist $\kappa>1$ and $C>0$ so that
\begin{equation}
(n+1) \nu(n+1) \le n^2 + C \kappa^n +
 \sum_{\genfrac{}{}{0pt}{}{j|(n+1)}{1\le j\le n}} j \nu(j)
\end{equation}
for all $n\ge 1$,
then there exists $C'$ so that $n\, \nu(n) \le C' \kappa^n$
for all $n\ge 1$.
\end{lemma}

\begin{proof} First choose $N$ so that
\begin{equation}
\frac{n}{\kappa^{(n+1)/2}} < \frac{2}{3} 
\end{equation}
if $n\ge N$.  Then choose $C'$ so that
\begin{enumerate}
\item\label{item:1} $C'>3 C$,
\item\label{item:2} $n^2 < \frac13 C' \kappa^n$ for all $n$, and
\item\label{item:3} $n\, \nu(n) < C' \kappa^n$ for $n\le N$.
\end{enumerate}
Note that (\ref{item:2}) uses only the fact that $\kappa>1$,
and (\ref{item:3}) uses only that $\kappa>0$ and $N$ is finite.

Now we prove the desired estimate by induction.
Suppose $n\, \nu(n) \le C' \kappa^n$ for $n\le M$, where $M > N$,
and suppose $n = M + 1$.
Using the first two conditions on $C'$,
the induction hypothesis, and the fact that all
proper divisors of $n+1$ are at most $(n+1)/2$, we have
\begin{align}
(n+1)\nu(n+1) < \mathstrut & \frac{1}{3} C' \kappa^n +
	\frac{1}{3} C' \kappa^n +
	\sum_{j\le (n+1)/2} j\nu(j) \cr
\le \mathstrut & \frac{2}{3} C' \kappa^n +
        \sum_{j\le (n+1)/2}  C' \kappa^j \cr
\le \mathstrut & \frac{2}{3} C' \kappa^n +
        \frac{n+1}{2} C' \kappa^{(n+1)/2} \cr
=\mathstrut & \frac{2}{3} C' \kappa^n +
        C' \kappa^{n+1} \frac{n+1}{2 \kappa^{(n+1)/2} } \cr
\le \mathstrut & C' \kappa^{n+1} .
\end{align}
The last inequality follows from $n>N$ and the choice of~$N$.
That completes the proof of Lemma~\ref{lem:boundnu}.
\end{proof}

\subsection{Proof of Theorem~\ref{thm:denominator0}}
\label{proofoftheoremandcorollary}

We describe how to modify the the proof of Theorem~\ref{thm:around0}
to give a proof of  Theorem~\ref{thm:denominator0}.

\begin{proof}[Proof of Theorem~\ref{thm:denominator0}] Let 
\begin{equation}
h(z)=
\prod_{j=1}^J \left(\frac{1+ \xi_j z^j}{1+ \xi_j R^j z^j}\right)^{\nu(j)}
\frac{1+\eta_j z^j}{1+\eta_j R^j z^j}.
\end{equation}
Note that
\begin{equation}
\frac{h'}{h}(z) = \frac{g'}{g}(z) - R \frac{g'}{g}(R z),
\end{equation}
where $g$ is the function 
\eqref{eqn:g} appearing in the proof of Theorem~\ref{thm:around0}.
Writing
\begin{equation}
\frac{h'}{h}(z)  = \sum_{k=0}^\infty c_k z^k
\end{equation}
we have
\begin{equation}
c_k = (1-R^{k+1}) b_k,
\end{equation}
where $b_k$ are the Taylor series coefficients of $g'/g$ given in~\eqref{eqn:gpg}.

Thus, when matching the coefficients of $h'/h$ and $f'/f$, everything goes
as before if in each equation we replace $b_k$ by $c_k$ and 
$a_k$ by $a_k/(1-R^{k+1})$.  So the choices of $\xi_k$, $\eta_k$, and $\nu(k)$
follow the same steps.  The final step of bounding $\nu(K)$ involves
replacing inequality~\eqref{eqn:nukinequality} by
\begin{equation}
(K+1)\nu(K+1) \le  K+1+ \frac{|a_K|}{1-R^{K+1}} 
+ \sum_{\genfrac{}{}{0pt}{}{j|(K+1)}{1\le j\le K}}
j \nu(j) + j .
\end{equation}
But that implies the bound we need on $\nu(j)$ because the only fact we used
about $a_k$ is $|a_k|\le C \kappa^k$ for some~$C>0$.

This completes the proof of Theorem~\ref{thm:denominator0}. \end{proof}

We turn to the proof of 
Corollary~\ref{cor:blaschke0}.

\begin{proof}[Proof of Corollary~\ref{cor:blaschke0}]
In Theorem~\ref{thm:denominator0}, let $R=r^2$ and set $g(z)=f(z/r)$, so
\begin{equation*}
g(r z) = f(z)=
a_0 \prod_{j=1}^\infty
\left(\frac{1+ \xi_j z^j}{1+  {\xi}_j r^{2j} z^j}\right)^{\nu(j)}
\frac{1+\eta_j z^j}{1+\eta_j r^{2j} z^j}.
\end{equation*}
By Theorem 2.2, for each $\delta > 0$ we may choose $N$ so that
\begin{equation*}
\sup_{|z| < 1 -\delta} \left|g(r z) -
a_0 \prod_{j=1}^N
\left(\frac{1+ \xi_j z^j}{1+  {\xi}_j r^{2j} z^j}\right)^{\nu(j)}
\frac{1+\eta_j z^j}{1+\eta_j r^{2j} z^j}\right| < \varepsilon.
\end{equation*}

It remains to rearrange the above product to recognize it as
a Blaschke product.
Letting $w = r z$ we have
\begin{equation*}
\sup_{|w| < r(1 -\delta)} \left|g(w) -
a_0 \prod_{j=1}^N
\left(\frac{1+ r^{-j} \xi_j w^j}{1+  {\xi}_j {r}^j w^j}\right)^{\nu(j)}
\frac{1+r^{-j} \eta_j w^j}{1+\eta_j {r^j} w^j}\right| < \varepsilon.
\end{equation*}

Thus, \begin{equation*}
\sup_{|w| < r(1 -\delta)} \left|g(w) -
a_0 \prod_{j=1}^N
r^{-2j} \xi_j  \eta_j \left(\frac{r^{j} \overline{\xi_j}+ w^j}{1+  {\xi}_j {r}^j w^j}\right)^{\nu(j)}
\frac{r^j \overline{\eta_j}+ w^j}{1+\eta_j {r^j} w^j}\right| < \varepsilon.
\end{equation*}

Letting $\alpha_j = r^j \overline{\xi_j}$ and
$\beta_j =  r^j \overline{\eta_j}$,  we have 
\begin{equation}
\sup_{|w| < r(1 - \delta)}
 \left|g(w) -  \left(a_0 \prod_{j = 1}^N r^{-2j} \xi_j \eta_j\right)
     C(w)\right| < \varepsilon ,
\end{equation}
where $\displaystyle C(w) = \prod_{j = 1}^N \left(\frac{\alpha_j +
w^j}{1 + \overline{\alpha_j} w^j}\right)^{\nu(j)}\left(\frac{\beta_j +
w^j}{1 + \overline{\beta_j} w^j}\right)$. Since each factor of $C$ is a
M\"obius transformation composed with $w^j$, each factor is a Blaschke
product and therefore $C$ is a Blaschke product as well.
\end{proof}

\section{Corollaries of Theorem~\ref{thm:around0}}\label{sec:corollaries}

We deduce some corollaries.

\subsection{Approximation on pseudohyperbolic discs}\label{sec:pseudohyperbolic}

As a corollary to Theorem~\ref{thm:around0} we prove a result
about approximating on other discs contained in the unit disc. We
recall first that the pseudohyperbolic distance between two points $z$
and $w$ in $\mathbb{D}$ is defined to be the distance 
\begin{equation}
\rho(z, w) =
\left|\frac{z - w}{1 - \overline{w} z}\right|.
\end{equation}
For $a \in \mathbb{D}$
and $r$ with $0 < r < 1$ we let $D_\rho(a, r) = \{z: \rho(a, z) <
r\}$. Given a Euclidean disc $D(a_0, r_0)$, we may rotate it  so that the
center lies on the positive real axis, and let $x$ and $y$, with $|x|
< y$,  denote the points in which the bounding circle $\mathcal{C}$
intersects the real line.  Let $\alpha_a$ be the M\"obius function
\begin{equation}
\alpha_a(z) = \frac{z + a}{1 + \overline{a} z}
\end{equation}
and let 
$R = \frac{1 + xy}{x + y}$.
Then $R > 1$ and if $a = R - \sqrt{R^2 - 1}$ then $r = -\alpha_a^{-1}(x) =
\alpha_a^{-1}(y)$. Since $a$ is real, $\alpha_a^{-1}$ maps $\mathcal{C}$
onto a circle $\mathcal{C}_1$ passing through $r$ and $-r$ and since
the real line is orthogonal to $\mathcal{C}$, the real line must
be orthogonal to $\mathcal{C}_1$. Therefore $\alpha_a^{-1}$ maps
$\mathcal{C}$ onto $\{z: |z| = r\}$. Thus,  the disc $D(a_0,r_0)$ is
rotation of a pseudohyperbolic disc $D_\rho(a, r)$
for some $a$,~$r$.  This means that
\begin{equation}
D_\rho(a, r) = \alpha_a(D(0,r))
\label{eqn:mobiusdisc}
\end{equation}

For basic information about automorphisms of the disc, see Garnett~\cite{G}.

\begin{theorem}
Let $f$ be a function that is analytic and nonvanishing in a
neighborhood of the
disc $\overline{D(a_0, r_0)}\subset \D$. Then $f$ can be uniformly
approximated on $D(a_0, r_0)$ by a polynomial with zeros on the unit
circle.
\end{theorem}

\begin{proof}

Suppose $f$ has no zeros in a neighborhood of the closure of a
pseudohyperbolic disc
$D_\rho(a, r)$. Then $f \circ {\alpha_a}$ has no zeros
in a neighborhood of the disc $D(0, r)$.
By Theorem~\ref{thm:around0},
there is a
polynomial $p$ with zeros on the unit circle such that
\begin{equation}
\|f \circ {\alpha_a} - p\|_{D(0, r)} < \varepsilon.
\end{equation}
Therefore by \eqref{eqn:mobiusdisc} and a change of variables,
\begin{equation}
\|f - p \circ {\alpha_a}^{-1}\|_{D_\rho(a,
r)} < \varepsilon.
\label{eqn:pestimate}
\end{equation}

Now, letting $z_1, \ldots, z_N$ denote the zeros of $p$ on the unit circle
we see that 
\begin{equation}
p \circ \alpha_a^{-1} (z) = \prod_{j = 1}^N ({\alpha_a}^{-1}(z)
- z_j).
\end{equation}
This is a rational function with poles outside the (closed)
unit disc and zeros at ${\alpha_a}(z_j)$ for $j = 1, \ldots, N$. Thus,
the zeros of this rational function also lie on the unit circle.

Now
choose
$s < 1$ so that $D_\rho(a, r) \subset D(0, s)$.
Since $p \circ \alpha_a^{-1} (z)$ is analytic and nonvanishing in
a neighborhood of $D(0, s)$, we can apply Theorem~\ref{thm:around0}
again
to get a polynomial $q$ so that
\begin{equation}
\|q - p \circ {\alpha_a}^{-1}\|_{D(0, s)}
< \varepsilon,
\end{equation}
which implies 
\begin{equation}
\|q - p \circ {\alpha_a}^{-1}\|_{D_\rho(a, r)}
< \varepsilon.
\label{eqn:qestimate}
\end{equation}
Combining \eqref{eqn:pestimate} and \eqref{eqn:qestimate}
gives $\|f - q\|_{D_\rho(a, r)} < 2\varepsilon$,
as required.

\end{proof}

There is a Blaschke product version of this result that can be
obtained in a similar, but simpler manner; that is, if we use
Corollary~\ref{cor:blaschke0} in place of Theorem~\ref{thm:around0}
and note that $B \circ \alpha_a^{-1}$ is a finite Blaschke product (see
\cite[p. 6]{G}) with zeros on the boundary of $D_\rho(a, r)$ whenever
the zeros of $B$ lie on $\{z: |z| = r\}$, we obtain the following result.

\begin{theorem} Let $f$ be a function that is analytic and nonvanishing
in a neighborhood of 
the disc     $\overline{D(a_0, r_0)} \subset \mathbb{D}$. Then
$f$ can be uniformly approximated on $D(a_0, r_0-\delta)$ by a 
constant times a Blaschke product
with zeros on the circle $\{z: |z - a_0| = r_0\}$.
\end{theorem}

From these results, we obtain a corollary about functions with
zeros. 
For $0 < p < \infty$ and $f$ an analytic function on 
$\mathbb{D}$, we say that $f \in H^p$ if
\begin{equation}
\sup_r \frac{1}{2 \pi} \int |f(r e^{i \theta})|^p d \theta =
\|f\|_{H^p}^p < \infty.
\end{equation}
It is well known that given a nonzero
function $f \in H^p$ the zero sequence of $f$, denoted $(z_n)$,
is a Blaschke sequence. Letting $C_1$ denote the (possibly infinite)
Blaschke product with zeros $(z_n)$ there exists a function $g$ that is
analytic on $\mathbb{D}$ and has no zeros in $\mathbb{D}$ such that $f =
C_1 g$. Applying the previous theorem to $g$ we obtain the following.

\begin{corollary} Let $0 < p < \infty$ and let $f \in H^p$.
If $\overline{D(a_0, r_0)} \subset \mathbb{D}$ then $f$ can be uniformly
approximated on $D(a_0, r_0-\delta)$ by functions of the
form $c_0 C_1 C_2$ where $c_0$ is a constant, 
$C_1$
is the Blaschke factor of $f$, and $C_2$ is a Blaschke product with
zeros on the circle $\{z: |z - a_0| = r_0\}$.
\end{corollary}

\section{The Riemann zeta function and random matrix theory}\label{sec:rmt}

We describe a new approach and another motivation for studying approximations
by polynomials with zeros on the unit circle.  This involves a combination
of a universality result for the Riemann zeta function, and the
connection between the zeta function and random matrices.

We recall Voronin's universality result.

\begin{theorem}(Voronin~\cite{V,St}) Let $0<r<\frac14$ and suppose
$g$ is a nonvanishing continuous function on the disc $|s|\le r$
which is analytic in the interior.  Then for any $\varepsilon>0$,
\begin{equation}
\liminf_{T\to \infty} \frac{1}{T}\, \mathrm{meas}  \left\{
\tau\in [0,T]\ :\ \max_{|s|<r}
|\zeta(\tfrac34 + i\tau +s)  - g(s)|<\varepsilon \right\} > 0.
\end{equation}
\end{theorem}

Here $\zeta(s)$ is the Riemann zeta function, defined by
\begin{equation}\label{eqn:zetaDS}
\zeta(s)=\sum_{n=1}^\infty \frac{1}{n^s}
\end{equation}
for $\Re(s)>1$. 
The standard reference is Titchmarsh~\cite{T}
and we merely cite a few facts that are relevant to our discussion here.
The zeta function continues to a meromorphic
function with a single simple pole at~$s=1$.
The functional equation relates $\zeta(s)$ to $\zeta(1-s)$.
Combining the functional equation with~\eqref{eqn:zetaDS}
shows that the most interesting behavior of the
zeta function will occur in the
``critical strip'' $0< \Re(s) < 1$.  Voronin's 
theorem says that the zeta function is universal in the
right half of the critical strip.  Furthermore, any given 
nonvanishing analytic function on a disc of radius $<\frac14$
is closely approximated by a \emph{positive proportion} of shifts
of $\zeta(\frac34 + i s)$.

The Riemann Hypothesis is the conjecture that the zeros
of the zeta function in the critical strip all lie
on the ``critical line'' $\Re(s)=\frac12$.
Recently there has been significant progress in
understanding (i.e., conjecturing) the behavior of the zeta function
in the critical strip
by modeling the zeta function by the characteristic
polynomial of a random unitary matrix~\cite{KS,CFKRS}.
Specifically, $\zeta(\frac12 + i T + z)$, for $\Re(z)>0$, is modeled
statistically
by $\Lambda(e^{-z})$, where $\Lambda$ is the characteristic
polynomial of a random matrix from the unitary group~$U(N)$,
chosen uniformly with respect to Haar measure, where
$N$ is approximately~$\log(T/2\pi)$.

Thus, Voronin's theorem and the above discussion suggests the
theorem that characteristic polynomials
of unitary matrices (ie, polynomials having all zeros on the
unit circle) can approximate nonvanishing functions 
in the unit disc.  But not quite: the normalization
used in analytic number theory has $\Lambda(0)=1$, corresponding
to the fact that $\lim_{\Re(s)\to\infty} \zeta(s)=1$.
Thus, characteristic polynomials can approximate nonvanishing functions on
$D(a_0,r_0)\subset \mathbb{D}$ provided $0\not\in \overline{D(a_0,r_0)}$.
This of course follows from the results described at the beginning
of Section~\ref{sec:around0}, but we suggest that it would be interesting
to give a purely random matrix proof.
That is, to understand the distribution
of
\begin{equation}\label{eqn:rmttuple}
\left(\frac{\Lambda'}{\Lambda}(x), \ldots, \frac{\Lambda^{(n)}}{\Lambda}(x)
\right), 
\end{equation}
with $0< x <1$ fixed,
for $\Lambda$ the characteristic polynomial of a random matrix
in~$U(N)$ as $N\to\infty$.  
The approximation theorems show that each $z\in \C^n$ is of the above form
if $N$ is sufficiently large.
But a purely random matrix calculation, providing a probability
distribution for~\eqref{eqn:rmttuple}, could give yet another proof.

It would also be interesting to prove a random matrix analogue
of the fact that a positive density of shifts of the zeta function
approximate a given function.  
Suppose $f$ is a nonvanishing analytic function on the unit disc with
$f(0)=1$, and suppose $0<r<1$ and 
$\varepsilon>0$ are given.
If $e^N$ matrices $U\in U(N)$ are chosen randomly with respect to
Haar measure, the problem is to determine the probability that at least one
of those $e^N$ matrices satisfies
\begin{equation}
|\det(I-U^* z) - f(z)|<\varepsilon \ \ \text{ for all } \ \ |z|<r .
\end{equation}
See~\cite{FGH} for an explanation of why $e^N$ matrices are
chosen from~$N(U)$.  Voronin's universality result suggests
that if $N$ is sufficiently large then the above probability
is positive and bounded below independent of~$N$.

If this random matrix approach is successful, it will produce polynomials
of a very different form than those constructed in the proof of
Theorem~\ref{thm:around0}.  That proof involved polynomials with high
multiplicity in their zeros.  But the characteristic polynomial of
a random unitary matrix has, with probability~1, only simple
zeros, and those zeros tend to be very evenly spaced on the
unit circle.

\end{document}